\documentclass{amsart}
\usepackage{amsmath,amssymb,graphicx,mathrsfs,verbatim}

\newtheorem{thm}{Theorem}
\newtheorem{conj}{Conjecture}[thm]
\newtheorem{prop}{Proposition}[thm]
\newtheorem{lemma}{Lemma}[thm]

\begin{document}

\title{Large $2$-coloured matchings in $3$-coloured complete hypergraphs}
\author[T. Terpai]{Tam\'as Terpai}
\address{Universit\'e de Gen\`eve\\
rue du Li\`evre 2-4, Case postale 64\\
1211 Gen\`eve 4, Switzerland}
\email{terpai@math.elte.hu}
\thanks{Supported by OTKA grant NK 81203}

\subjclass[2000]{05C55}
\keywords{generalized Ramsey theory on hypergraphs}

\begin{abstract}
We prove the generalized Ramsey-type result on large $2$-coloured matchings
in a $3$-coloured complete $3$-uniform hypergraph, supporting a conjecture
by A. Gy\'arf\'as.
\end{abstract}

\maketitle

\section{Introduction and statement of result}
In \cite{gyarfas}, the authors consider generalisations of Ramsey-type
problems where the goal is not to find a monochromatic subgraph, but a
subgraph that uses ``few'' colours. In particular, the following theorem is
proven:

\begin{thm}[{\cite[Theorem 13]{gyarfas}}] For $k \geq 1$, in every
$3$-colouring of a complete graph with $f(k)=\left \lfloor \frac{7k-1}{3}
\right \rfloor$
vertices there is a $2$-coloured matching of size $k$. This is sharp for
every $k \geq 2$, i.e. there is a $3$-colouring of $K_{f(k)-1}$ that does
not contain a $2$-coloured matching of size $k$.
\end{thm}

The example that shows the sharpness of the estimate is close to the
colouring obtained by first colouring the vertices with the available
colours in proportion close to $1:2:4$ and then colouring the edges by the
lowest index colour among its endpoints. The analogous question and
construction make sense in the case of complete hypergraphs instead of
$K_n$. At the $1.$ Eml\'ekt\'abla workshop held at Gy\"ongy\"ostarj\'an in
July $2010$, the first nontrivial case of this question (with $3$-uniform
hypergraphs and $3$ colours) was considered. The best known construction
in this case is based on the proportion $1:3:9$, and leads to the following
conjecture by A. Gy\'arf\'as:

\begin{conj}
For any $t$-colouring of the complete $r$-uniform hypergraph on
$$
n \geq kr + \left\lfloor \frac{(k-1)(t-s)}{1+r+\dots+r^{s-1}} \right\rfloor
$$
vertices, there exists a $s$-coloured matching of size $k$.
\end{conj}

While it is known that the conjecture fails for e.g. $t=6$ and $s=2$, several
particular cases are open. We consider here only $t=3$, $r=3$ and $s=2$, in
which case the conjecture has the form

\begin{thm}\label{thm:preMain}
For any $3$-colouring of the complete $3$-uniform hypergraph on
$$
n \geq 3k+\left\lfloor  \frac{k-1}{4} \right\rfloor
$$
vertices, there exists a $2$-coloured matching of size $k$.
\end{thm}

The case $k=4$ (the first case that is not a trivial consequence of the
existing results for the monochromatic problem, see e.g. \cite{Alonetal})
was confirmed at the workshop by a team consisting of N. Bushaw, A.
Gy\'arfas, D. Gerbner, L. Merchant, D. Piguet, A. Riet, D. Vu and the
author:

\begin{thm}[\cite{kozos}]\label{thm:kozos}
For any $3$-colouring of the complete $3$-uniform hypergraph on $12$
vertices there exists a perfect matching that uses at most $2$ colours.
\end{thm}

In this paper, we prove Theorem \ref{thm:preMain} in the following
equivalent form:

\begin{thm}\label{thm:main}
For any $3$-colouring of the complete $3$-uniform hypergraph on $n$
vertices, there exists a $2$-coloured matching of size
\begin{equation}\label{eqn:n2k}
m(n)=\left\lfloor \frac{4(n+1)}{13} \right\rfloor.
\end{equation}
\end{thm}

It is easy to check that these indeed are formulations of the same result as
$$
n = 3k+\left\lfloor \frac{k-1}{4} \right\rfloor = \left\lfloor \frac{13k-1}{4}
\right\rfloor
$$
is the smallest integer for which $\left\lfloor \frac{4(n+1)}{13}
\right\rfloor \geq k$ holds.

\section{Proof}
In the proof, the set of vertices of the hypergraph will be denoted by $V$,
the colouring will be a function $c: {V \choose 3} \to \{ 1, 2, 3 \}$, and
$\alpha$, $\beta$, $\gamma$ will be an arbitrary permutation of the colours
$1$, $2$, $3$. The colours are shifted cyclically, e.g. if $\alpha=3$, then
$\alpha+1$ denotes the colour $1$ and if $\alpha=1$, then $\alpha−1$ denotes
the colour $3$. A matching on $n$ vertices is {\it near perfect}, if it has
size $\lfloor n/3 \rfloor$.
\par
We call a sextuple $A$ of points {\it $\alpha$-dominated} for a colour
$\alpha=1,2,3$, if for all splittings $A = B_1 \cup B_2$ into two disjoint
triples at least one of $c(B_1)=\alpha$ or $c(B_2)=\alpha$ holds. If $A$ is
not $\alpha$-dominated for any $\alpha$, we call it {\it universal}.
Similarly, we call a set $X$ of $13$ points universal if it admits a near
perfect matching in any pair of colours.
\par
The proof proceeds by taking a maximal set of disjoint universal sets $A_1,
\dots, A_l$, $X_1, \dots, X_m$ with $\vert A_i \vert = 6$ and $\vert X_j \vert
=13$. If we can now construct a $2$-colour matching on $W = V \setminus (A_1
\cup \dots \cup X_m)$ of the size $m(\vert W \vert)$, then we can extend it
by the appropriately coloured near perfect matchings in the universal sets
$A_i$ and $X_j$ and keep the size of the matching at least $m(\vert V
\vert)$. Indeed, in the case of an $A_i$ decreasing $n$ by $6$ decreases
$m(n)$ by at most $\lceil 4 \cdot 6 /13 \rceil = 2$ and in the case of an
$X_j$ decreasing $n$ by $13$ decreases $m(n)$ by $4$. Thus by switching to
$W$ we may assume that there are no universal sets of size $6$ or $13$, and
the resulting structural properties of the colouring will give us the
necessary large $2$-colour matching.
\par
If a vertex sextuple is $\alpha$-dominated, and there are splittings of it
into hyperedges of colours $\alpha$ and $\alpha+1$ as well as into those of
colours $\alpha$ and $\alpha+2$, we call this sextuple a {\it spread} in
colour $\alpha$, and the splittings are its {\it demonstration splittings}.
Depending of whether the hyperedges of colour $\alpha$ in the demonstration
splittings overlap in $1$ or $2$ vertices, we assign the spread (with a
fixed demonstration splitting implied) a level of $1$ or $2$ respectively.

\begin{lemma}\label{thm:coupling}
Assume that there are two disjoint spreads $A$ in colour $\alpha$ and $B$ in
colour $\beta$ such that $\alpha \neq \beta$, and let $v$ be an arbitrary
vertex in the complement $V \setminus (A \cup B)$ of their union. Then the
following property holds for $X=A$ or for $X=B$ (or both): if we substitute
$v$ for any vertex of the dominantly coloured triple in either
demonstration splitting of the spread $X$, the colour of that triple stays
the same, the dominating colour of $X$.
\end{lemma}

%% \begin{comment}
\par
\begin{figure}
\centering
\includegraphics[width=12.5cm]{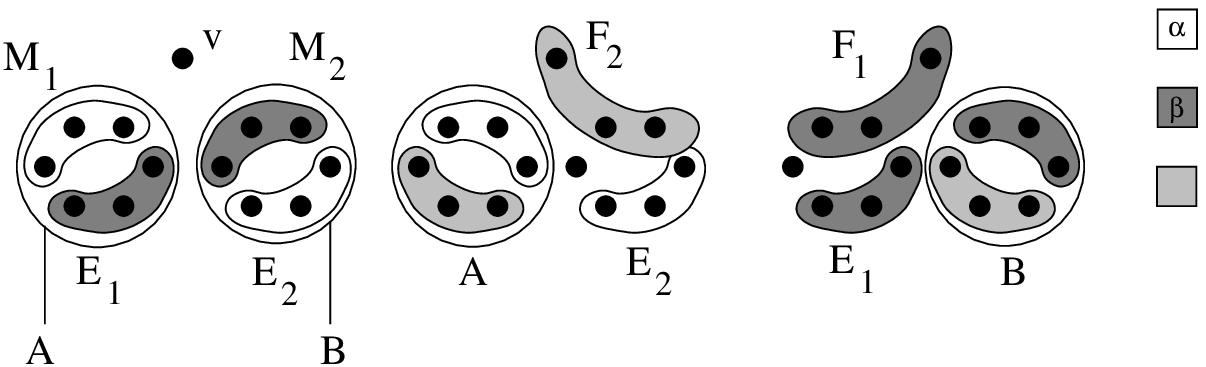}
\caption{If the substitution of a single vertex $v$ can change the colour of a
dominant triple in two spreads of different colour, we have a universal set
of size $13$. Example: case when $c(E_1)=\beta$ and $c(E_2)=\alpha$.}
\label{fig:spread_match}
\end{figure}
\par
%% \end{comment}

\begin{proof}
Indirectly assume that $A=M_1 \cup E_1$ and $B=M_2 \cup E_2$ are splittings
which v ``spoils''. That is, $M_1$ has the dominant colour $\alpha$ of $A$
and $M_1 \cup \{ v \}$ contains a triple $F_1$ of colour different from $c(M_1) =
\alpha$, and analogously $c(M_2) = \beta$ and $M_2 \cup \{ v \}$ contains a
triple $F_2$ of colour different from $\beta$ (see example on Figure
\ref{fig:spread_match}). Then $A \cup B \cup \{ v \}$
is a universal set of size $13$. Indeed, it has a matching of size $4$ that
contains only colours $\alpha$ and $\beta$ since both $A$ and $B$ admit
perfect matchings in these colours. It also has a matching of size $4$ that
avoids the colour $\alpha$: the spread $B$ has such a matching of size $2$,
the triple $E_1$ has a colour different from $\alpha$, and the remainder $M_1
\cup \{ v \}$ contains $F_1$, also a triple of a colour different from
$\alpha$. The same argument with $A$ and $B$ reversed produces a near
perfect matching that avoids the colour $\beta$, proving our claim and
arriving at the contradiction that proves the lemma.
\end{proof}

This coupling property implies a very rigid structure of the colouring:

\begin{prop}\label{thm:twoSpreads}
If there is a pair of disjoint spreads in two different colours, then there
is a nearly perfect matching avoiding one colour.
\end{prop}

\begin{proof}
Without loss of generality we may assume that the two colours are $1$ and
$2$. Let the two spreads be $A^{(1)}$ (colour $1$) and $A^{(2)}$ (colour
$2$), and out of all disjoint pairs of spreads of colours $1$ and $2$ this
one contains the most level $2$ spreads. Then each of them is either level
$2$ or it is level $1$ and there are no level $2$ spreads of their colour
that would be disjoint from the other spread.
\par
In both $A^{(1)}$ and $A^{(2)}$, fix two demonstration splittings
$$A^{(i)} = M_+^{(i)} \cup P^{(i)} = M_-^{(i)} \cup N^{(i)}$$
such that $c(M^{(i)}_+) = c(M^{(i)}_-) = i$, $c(P^{(i)}) = i+1$ and
$c(N^{(i)})=i-1$. Depending on the level of $A^{(i)}$, we can label the
vertices of $A^{(i)} = \{ a^{(i)}_1, \dots,a^{(i)}_6\}$ to satisfy the
following equalities:
\begin{itemize}
\item in case of level $1$:
\begin{align*}
M^{(i)}_+= \{ v^{(i)}_1,v^{(i)}_2,v^{(i)}_3 \} & \qquad & P^{(i)}=
\{ v^{(i)}_4,v^{(i)}_5,v^{(i)}_6 \} \\
M^{(i)}_-= \{ v^{(i)}_1,v^{(i)}_4,v^{(i)}_5 \} & \qquad & N^{(i)}=
\{ v^{(i)}_2,v^{(i)}_3,v^{(i)}_6 \}
\end{align*}
We will call the vertex $v^{(i)}_1$ the {\it dominating} vertex and the rest
of the vertices the {\it core} vertices.

\item in case of level $2$:
\begin{align*}
M^{(i)}_+= \{ v^{(i)}_1,v^{(i)}_2,v^{(i)}_3 \} & \qquad & P^{(i)}=
\{ v^{(i)}_4,v^{(i)}_5,v^{(i)}_6 \} \\
M^{(i)}_-= \{ v^{(i)}_1,v^{(i)}_2,v^{(i)}_4 \} & \qquad & N^{(i)}=
\{ v^{(i)}_3,v^{(i)}_5,v^{(i)}_6 \}
\end{align*}
We will call the vertices $v^{(i)}_1$ and $v^{(i)}_2$ the {\it dominating}
vertices and the rest of the vertices the {\it core} vertices.
\end{itemize}
In both cases, $D^{(i)}$ will denote the set of the dominating vertices and
$C^{(i)}$ will denote the set of the core vertices. A pair of vertices will
be called {\it critical}, if they are contained in either $M^{(i)}_+$ or
$M^{(i)}_-$.
\par
We choose sets $\hat A^{(1)}$ and $\hat A^{(2)}$ to be a maximal disjoint
pair of sets satisfying the following properties:
\begin{itemize}
\item $D(i) \subseteq \hat A^{(i)} \subseteq V \setminus \left( C^{(i)} \cup
A^{(3-i)} \right)$ for $i=1,2$.
\item For any subset $D$ of $\hat A^(i)$ of size $\vert D \vert = \vert
D^{(i)} \vert$, the triples of the sextuple $C^{(i)} \cup D = \left( A^{(i)}
\setminus D^{(i)} \right) \cup D$ complementary to $P^{(i)}$ and $N^{(i)}$
have colour $i$.
\item For any subset $D$ of $\hat A^{(i)}$ of size $\vert D \vert = \vert
D^{(i)} \vert$, any pair of vertices $(u,v) \in V$ that is covered by the
complement of either $P^{(i)}$ or $N^{(i)}$ in the sextuple $C^{(i)} \cup D
= \left( A^{(i)} \setminus D^{(i)} \right) \cup D$, and any vertex $w \in
\hat A^{(i)} \setminus D$, the triple $\{ u,v,w \}$ has colour $i$.
\end{itemize}
That is, $\hat A^{(i)}$ is a maximal set of vertices (outside of $\hat
A^{(3-i)}$) extending the set of dominant vertices of $A^{(i)}$ such that we
can switch the dominant vertices of $A^{(i)}$ with any two vertices of $\hat
A^{(i)}$ and still be unable to change the colour of the dominant triples of
the modified sextuple by a single vertex change within the set $\hat
A^{(i)}$. Such sets exist (for example, $D^{(i)}$ satisfies the requirements
for $\hat A^{(i)}$), and their total size is bounded by $\vert V \vert$, so
we can choose a maximal pair.
\par
We claim that the sets $\hat A^{(1)} \cup \hat A^{(2)}$ cover $V \setminus
(C^{(1)} \cup C^{(2)})$. Indeed, assume that there is a vertex $w \in V
\setminus \left( C^{(1)} \cup C^{(2)} \cup \hat A^{(1)} \cup \hat A^{(2)}
\right)$ such that it cannot be added to either $\hat A^{(1)}$ or $\hat
A^{(2)}$ without violating their defining properties. This means that for
$i=1,2$ we can switch the vertices in $D^{(i)}$ to some other vertices in
$\hat A^{(i)}$ in such a way that for the resulting spread $\tilde A^{(i)}$
there is a pair of vertices $(u^{(i)},v^{(i)})$ of a dominating triple such
that
\begin{equation*} %% \label{eqn:bothWrong}
c(u^{(i)},v^{(i)},w) \neq i.
\end{equation*}
This contradicts Lemma \ref{thm:coupling} for the spreads $\tilde A^{(1)}$
and $\tilde A^{(2)}$ and the vertex $w$.
\par
Additionally, these sets are already easy to colour with $2$ colours:

\begin{lemma}\label{cliques}
The vertex set $\hat A^{(i)}$ is a clique of colour $i$.
\end{lemma}

\begin{proof}
We suppress for brevity the indices ${}^{(i)}$. If $A$ is level $2$, then
for any $\{ x,y,z \} \subseteq \hat A$ we have by definition of $\hat A$ the
property that $z$ forms triples of colour $i$ with all the critical vertex
pairs of $C \cup \{ x, y \}$, in particular, with $\{ x, y \}$, and we are
done.
\par
If $A$ is level $1$, recall first that we also assume that there are no
spreads of colour $i$ and level $2$ in $\hat A \cup C$. Indirectly assume
furthermore that there is a triple $X = \{ x_1, x_2, x_3 \} \subseteq \hat
A$ such that its colour is not $i$. For symmetry reasons it is enough to
check the case when $c(X)=i+1$. Then $P \cup X$ is covered by two disjoint
triples of colour $i+1$ and must be therefore $i+1$-dominated - otherwise
it would form a universal sextuple contrary to our assumptions. But $P
\setminus N = \{ v_4, v_5 \}$ is a critical pair of vertices and hence $\{
x_1, v_4, v_5 \}$ has colour $i$; therefore its complement $Y = \{ x_2, x_3,
v_6 \}$ has colour $i+1$. This implies that the sextuple $X \cup N = Y \cup
\{ x_1, v_2, v_3 \}$ can be split into colours $c(X)=i+1$ and $c(N)=i-1$ as
well as into colours $c(Y) = i+1$ and $c(\{ x_1, v_2, v_3 \}) = i$ (the set
$\{ v_2, v_3, x_1 \}$ is the complement of $P$ in $C \cup \{ x_1 \}$ with
$x_1 \in \hat A$), so this sextuple is $i+1$-dominated. Now use the fact
that $\{ x_2, v_3 \}$ is a critical pair of vertices in $C \cup \{ x_2 \}$
(it lies in the complement of $P$). This implies that $c(\{ x_2, x_3, v_3
\}) = i$, and consequently its complement in $X \cup N$ has colour $i+1$:
$$
c(\{ x_1, v_2, v_6 \}) = i+1.
$$
By definition of $\hat A$, the sextuple $\{ x_1 \} \cup C$ is $i$-dominant
as it cannot be dominant in any other colour. Hence the complement of $\{
x_1, v_2, v_6 \}$ in it has to have colour $i$:
$$
c(\{ v_3, v_4, v_5 \}) = i.
$$
Also, $\{ v_4, v_5 \}$ is a critical pair of vertices, so we have
$$
c(\{ x_1, v_4, v_5 \}) = i
$$
as well. But this means that $C \cup \{ x_1 \}$ is a level $2$ spread of
colour $i$ as evidenced by splitting into $\{ x_1, v_4, v_5 \} \cup N$
(colours $i$ and $i-1$ respectively) and into $\{ v_3, v_4, v_5 \} \cup \{
x_1, v_2, v_6 \}$ (colours $i$ and $i+1$ respectively) - a contradiction
with our initial assumption, hence $\hat A$ is indeed a clique of colour $i$
as claimed.
\end{proof}
\par
This also implies that $\hat A^{(1)} \cup M^{(1)}_+$ is a clique of colour
$1$ and $\hat A^{(2)} \cup M^{(2)}_-$ is a clique of colour $2$ (we are
adding a vertex or a critical pair of vertices to the appropriate $\hat
A^{(i)}$). Notice that their complement is the union of the $2$-coloured
hyperedge $P^{(1)}$ and the $1$-coloured hyperedge $N^{(2)}$.

\begin{lemma}\label{thm:cliques2matching}
If $U$ and $W$ are disjoint cliques of colours $1$ and $2$ respectively such
that $\vert U \vert \geq 3$ and $\vert W \vert \geq 3$, then there exists an
almost perfect matching in $U \cup W$ in colours $1$ and $2$.
\end{lemma}

\begin{proof}
If $\vert U \vert + \vert W \vert \operatorname{mod} 3=\vert U \vert
\operatorname{mod} 3 + \vert W \vert \operatorname{mod} 3$, that is, $\vert
U \vert \operatorname{mod} 3 + \vert W \vert \operatorname{mod} 3 \leq 2$,
then taking maximal disjoint sets of hyperedges in $U$ and $W$ separately
gives an almost perfect matching in colours $1$ and $2$.
\par
If this is not the case, then both $\vert U \vert \operatorname{mod} 3$ and
$\vert W \vert \operatorname{mod} 3$ are at least $1$ and at least one of
them is equal to $2$; without loss of generality, we may assume that $\vert
U \vert \equiv 2 \operatorname{mod} 3$. We claim that there is a hyperedge
$E \subset U \cup W$ of colour $1$ or $2$ with the property that $\vert U
\cap E \vert=2$. Indeed, assume indirectly that all triples intersecting $U$
in $2$ vertices and $W$ in $1$ vertex have colour $3$. Since $\vert U \vert
\geq 3$ and $\vert U \vert \operatorname{mod} 3 = 2$, we have $\vert U \vert
\geq 5$. Consider any four distinct vertices $u_1, u_2, u_3, u_4 \in U$ and
any two distinct vertices $w_1, w_2 \in W$. Then the set $X = \{ u_1, u_2,
u_3, u_4, w_1, w_2 \}$ is covered by the triples $\{ u_1, u_2, w_1 \}$ and
$\{ u_3, u_4, w_3 \}$, both of which have to have colour $3$. Hence $X$ can
only be $3$-dominated, consequently at least one of the members of the
matching $\{ u_1, w_1, w_2 \} \cup \{ u_2, u_3, u_4 \}$ has colour $3$.
But the triple $\{ u_2, u_3, u_4 \}$ lies in the clique $U$ and therefore
has colour $1$, so $c(\{ u_1, w_1, w_2 \}) = 3$. This implies that for any
choice of a vertex $w_3 \in W \setminus \{ w_1, w_2 \}$ we have on one hand
$$c(\{ u_1, w_1, w_2 \})=3\text{ and } c(\{ u_2, u_3, w_3 \})=3$$
due to the latter triple intersecting $U$ in $2$ vertices, and on the other
hand
$$c(\{ u_1, u_2, u_3 \})=1\text{ and } c(\{ w_1, w_2, w_3 \})=2$$
due to $U$ and $V$ being cliques. Hence $\{ u_1, u_2, u_3, w_1, w_2, w_3 \}$
would be a universal sextuple, a contradiction that proves our claim.
\par
Given a hyperedge $E \subset U \cup W$ of colour $1$ or $2$ with the
property that $\vert U \cap E \vert=2$, we can just add it to the union of
any maximal matching of $U \setminus E$ and any maximal matching of $W
\setminus E$ to get a nearly perfect matching of $U \cup W$ in colours $1$
and $2$.
\end{proof}

Applying Lemma \ref{thm:cliques2matching} to the cliques $\hat A^{(1)} \cup
M^{(1)}_+$ and $\hat A^{(2)} \cup M^{(2)}_-$ and adding the triples
$P^{(1)}$ and $N^{(2)}$ yields a near perfect matching in colours $1$ and
$2$ on $V$. This finishes the proof of Proposition \ref{thm:twoSpreads}.
\end{proof}
\par
Once we can exclude two disjoint spreads of different colours, we have two
possibilities: either there are no spreads at all, or there is a spread of,
say, colour $1$, and any spread in its complement is also of colour $1$. We
will also assume that $\vert V \vert \geq 9$ as otherwise the $2$-colour
condition is trivially fulfilled by any near perfect matching.
\par

%% \begin{comment}
\par
\begin{figure}
\centering
\includegraphics[width=3.5cm]{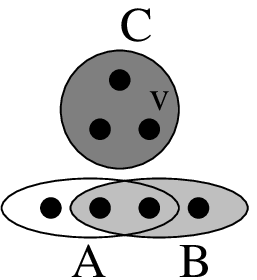}
\caption{If there are no spreads, then only two colours may be used.}
\label{fig:no_spreads}
\end{figure}
\par
%% \end{comment}

\par
{\bf Case 1: there are no spreads.} If there are no spreads, then no
sextuple can contain triples of all three colours: one of them would be
dominating, and any two instances of the other two colours could be chosen
to be $P$ and $N$ of a spread. We will first look for a pair of triples of
different colours that share two vertices, $c(A) \neq c(B)$, $\vert A \cap B
\vert =2$. If there are no such pairs, then all triples have the same colour
and any nearly perfect matching is monochromatic, we are done. If, on the
other hand, such triples $A$ and $B$ exist, we may assume without loss of
generality that $c(A)=1$ and $c(B)=2$. Could there be triples of colour $3$
(see Figure \ref{fig:no_spreads})? Any such triple $C$ would have to be
disjoint from $A \cup B$, because otherwise their union $A \cup B \cup C$
(together with any other vertex if it has only $5$ elements) would form a
sextuple of vertices that contains all the three colours. Then for any
vertex $v \in C$ the triple $T = (A \cap B) \cup \{ v \}$ is covered by both
$A \cup C$ (covering only triples of colour $1$ and $3$) and $B \cup C$
(covering only triples of colour $2$ and $3$) and therefore can only be of
colour $3$. But then $A \cup B \cup \{ v \}$ together with any other vertex
form a sextuple that contains triples of all three colours, $A$, $B$ and $T$
- a contradiction.
\par
Therefore in this case only two colours may be used at all, so any near
perfect matching automatically satisfies our desired condition.
\par
{\bf Case 2: there exists a spread} (of colour $1$, say). We first
investigate what happens if there are no spreads of other colour at all.
This results in a highly ordered structure:

\begin{prop}\label{thm:oneSpread}
If a colouring is such that all spreads are of colour $1$, then either
\begin{itemize}
\item there exists a near perfect matching avoiding colour $2$ or colour $3$, or
\item there are no triples of colour $1$ at all.
\end{itemize}
\end{prop}

\begin{proof}
First note that the condition on the spreads means that whenever a sextuple
contains triples of all three colours, it is $1$-dominated. In particular,
if a triple is covered by a disjoint union of a $2$-coloured and a
$3$-coloured triple, it cannot have colour $1$ - the union in question can
only be $2$- or $3$-dominated. We show that this statement can also be used
for non-disjoint pairs of triples of colours $2$ and $3$.

\begin{lemma}
Assume that all spreads in the colouring are of colour $1$. Then either
\begin{itemize}
\item there are no triples of colour $1$ covered by the union of a triple of
colour $2$ and a triple of colour $3$, or
\item there exists a near perfect matching avoiding colour $2$ as well as
one avoiding colour $3$.
\end{itemize}
\end{lemma}

\begin{proof}
Assume $A = \{ v_1, v_2, v_3 \}$ is a colour $1$ triple that is covered by
triples $B$ and $C$ of colours $2$ and $3$ respectively. At least one of
these has to cover $2$ vertices of $A$, so after a renumbering of colours,
triples and vertices we may assume that $B = \{ v_2, v_3, v_4 \}$ and $v_1
\in C$. We now have three cases for the situation of $C$ with respect to $A$
and $B$:

\begin{enumerate}

\item $C \cap B = \emptyset$. Then $B$ and $C$ are disjoint triples of
colour $2$ and $3$ respectively which cover $A$, a triple of colour $1$ - a
contradiction.

%% \begin{comment}
\par
\begin{figure}
\centering
\includegraphics[width=6cm]{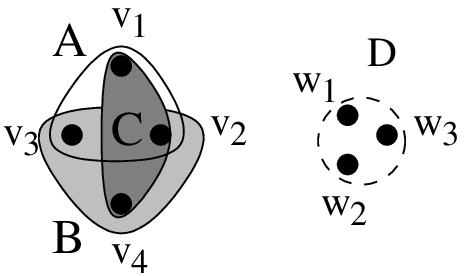}
\caption{Case $\vert C \cap B \vert = 2$.}
\label{fig:4rainbow}
\end{figure}
\par
%% \end{comment}

\item $C \cap B = \{ v_2, v_4 \}$ (or analogously $\{ v_3, v_4 \}$); that
is, $C$ is covered by $A \cup B$ (see Figure \ref{fig:4rainbow}). The union
$A \cup B = A \cup B \cup C$ has $4$ elements and contains all three
colours, so adding any pair of vertices $x$, $y$ makes it a $1$-dominated
sextuple. In this sextuple, the triples $\{ x, y, v_1 \}$ and $\{ x, y, v_3
\}$ have non-$1$-coloured complements, so they have to have colour $1$
themselves. Now assume there is a triple $D = \{ w_1, w_2, w_3 \}$ of colour
$2$ disjoint from $A \cup B$ (the case of $c(D)=3$ is similar). Then $D \cup
C$ is a disjoint union of a $2$-coloured triple and a $3$-coloured one, and
it covers the $1$-coloured triple $\{ w_1, w_2, v_1 \}$ - a contradiction.
Hence all triples disjoint from $A \cup B$ have colour $1$. Consequently we
can choose a near perfect matching either in colour $1$ only, or at will in
colours $1$ and $2$, or in colours $1$ and $3$ - if the total number of
vertices is congruent to $1$ or $2$ modulo $3$, we take a near perfect
matching in the complement of $A \cup B$ and add $A$, otherwise we take a
near perfect matching in the complement of $A \cup B$, add the triple $B$ or
$C$ depending on which colour out of $2$ and $3$ is wanted and match up the
remaining two vertices with either $v_1$ or $v_3$ (whichever is left out).

%% \begin{comment}
\par
\begin{figure}
\centering
\includegraphics[width=8cm]{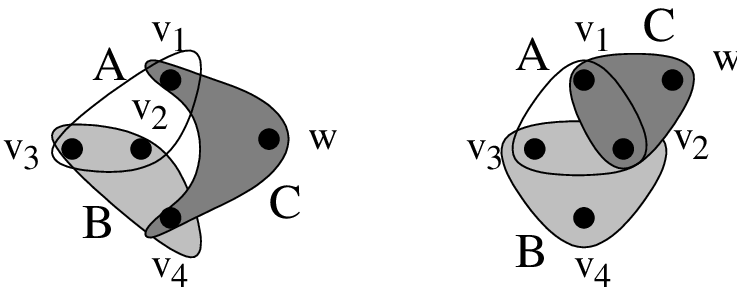}
\caption{Case $\vert C \cap B \vert = 1$.}
\label{fig:5rainbow}
\end{figure}
\par
%% \end{comment}

\item $\vert C \cap B \vert =1$; let $w$ denote the single vertex in $C
\setminus (A \cup B)$ (see Figure \ref{fig:5rainbow}). By the same argument
as before, for any vertex $x$ not in $A \cup B \cup C$ we have that $A \cup
B \cup C \cup \{ x \}$ is $1$-dominated, so the complements of the
non-$1$-coloured triples $B$ and $C$ must have colour $1$:
$$
c\left(\{ x \} \cup ((A \cup B \cup C) \setminus B)\right) = 1,
$$
$$
c\left(\{ x \} \cup ((A \cup B \cup C) \setminus C)\right) = 1.
$$
This makes it impossible to have triples of colour other than $1$ disjoint
from $A \cup B \cup C$, as they would cover a $1$-coloured triple together
with either $B$ or $C$ (whichever has the colour other from that of the
selected triple). Now taking a maximal matching outside $A \cup B \cup C$,
we can extend it to a near perfect matching avoiding the colour $2$ or the
colour $3$ as follows. If there are no vertices left outside the matching,
add $A$ to get a $1$-coloured matching. If there is $1$ vertex left, join it
to $(A \cup B \cup C) \setminus B$ and add $B$ to avoid the colour $3$; do
the same with $B$ and $C$ switched to avoid the colour $2$. Finally, if
there are $2$ vertices left, join them respectively to the disjoint vertex
pairs $(A \cup B \cup C) \setminus B$ and $(A \cup B \cup C) \setminus C$ to
obtain a matching of colour $1$.

\end{enumerate}
\end{proof}

Thus we may restrict our attention to the case when the union of a triple of
colour $2$ and a triple of colour $3$ cannot cover a triple of colour $1$,
even if they are not disjoint.
\par
We now try to find a vertex such that all triples containing it are of
colour $1$; we will call such a vertex {\it $1$-forcing}. If there are no
triples of colours $1$ and $2$ or $1$ and $3$ such that they intersect in
two vertices, then either there are no triples of colour $1$ - in which case
there is a near perfect matching in colours $2$ and $3$ - or there are no
triples of colour different from $1$ - in which case there is a near perfect
matching in colour $1$. Hence we might assume that there is a pair of
triples of the form $A=\{ v_1, v_2, v_3\}$, $B=\{ v_2, v_3, v_4 \}$ with
$c(A)=1$ and $c(B)=2$, say. By our previous lemma there are no triples of
colour $3$ that contain $v_1$. If there are no such triples disjoint from $A
\cup B$ either, then any near perfect matching containing $A$ or $B$ will be
$1$ and $2$-coloured. Assume therefore that there is a triple $C$ of colour
$c(C)=3$ in the complement of $A \cup B$. No triple covered by $B \cup C$
can have colour $1$, in particular the triple $D = \{ v_2, v_3, x \}$ with
some $x \in C$ has to have colour $2$ or $3$. If $c(D)=3$, then we can
repeat the argument with $A$ and $D$ instead of $A$ and $B$ to get that no
colour $2$ triples contain $v_1$ either, so $v_1$ is $1$-forcing. If
$c(D)=2$, then $A \cup C$ contains triples of all three colours and thus is
$1$-dominated, in particular the triple $E = \{ v_1 \} \cup (C \setminus \{
x \}) = (A \cup C) \setminus D$ has to have colour $1$ due to its complement
having colour $2$. In this case, the application of the same argument to $E$
and $C$ gives us the same result of no triples of colour $2$ containing
$v_1$ and $v_1$ is $1$-forcing again.
\par
Putting such a $1$-forcing vertex aside and repeating the procedure, we
end up with a set of $1$-forcing vertices and a remainder set where either
there are no triples of colour $1$ or there is a near perfect matching in
colours $1$ and $2$ or $3$. In the latter case, we can just complete the
matching with $1$-forcing vertices at will, so we assume now that the
remainder, denoted henceforth by $R$, only has triples of colours $2$ and
$3$.
\par

%% \begin{comment}
\par
\begin{figure}
\centering
\includegraphics[width=11cm]{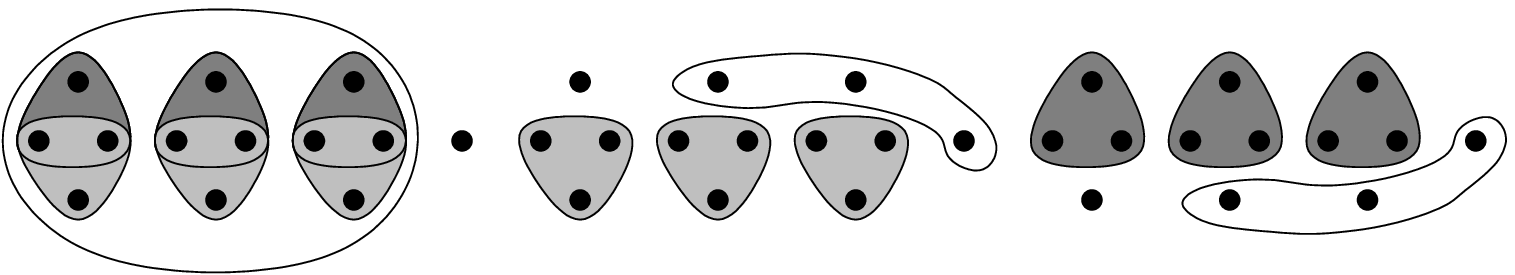}
\caption{A $1$-forcing vertex with $3$ disjoint ``neighbouring'' triples of
colours $2$ and $3$ implies the existence of a universal $13$-vertex set.}
\label{fig:3plus1}
\end{figure}
\par
%% \end{comment}

If $R=V$, then we get the second conclusion of our proposition, so we may
assume that there is at least one $1$-forcing vertex. If, moreover, $R$ had
three disjoint pairs of triples of colours $2$ and $3$ that intersect in $2$
vertices, we could add a $1$-forcing vertex and get a universal $13$-vertex
set (see Figure \ref{fig:3plus1}) - a contradiction. If there are no three
disjoint pairs like that, then after picking at most two of them the rest
(denoted by $R'$) has to be a clique, of colour $2$, say. We can then take a
$1$-forcing vertex and add to it those vertices of the triples of colour $3$
among the chosen pairs of colour $2$ and $3$ that are not covered by the
corresponding triples of colour $2$, and add another vertex from $R'$ if we
still don't have three vertices. Choose a near perfect matching from the
rest of $R$ containing the selected triples of colour $2$ and then cover the
rest with $1$-forcing vertices if there are any left. This yields a near
perfect matching in colours $1$ and $2$, and finishes the proof of the
proposition.
\end{proof}

Since in both cases of Proposition \ref{thm:oneSpread} we get a near perfect
matching in $2$ colours, we only need to consider the case where there exist
spreads of a different colour. By symmetry, assume that $U$ is a spread of
colour $1$ and $W$ is a spread of colour $2$. By Proposition
\ref{thm:twoSpreads}, we can apply Proposition \ref{thm:oneSpread} to $V
\setminus U$, so we either get a near perfect matching avoiding colour $2$
or $3$ or no triples of colour $1$ at all. In the first case, we can add one
of the demonstration splittings of $U$ to get a near perfect matching of $V$
avoiding either colour $2$ or colour $3$. The same argument of applying
Proposition \ref{thm:oneSpread} to $V \setminus W$ yields either a near
perfect matching on $V$ in $2$ colours or no triples of colour $2$ at all.
We may hence assume that we got the second result in both attempts, and the
colouring is such that all triples of colour $1$ intersect $U$ while all
triples of colour $2$ intersect $W$.
\par
We see that in such a setup, the vertex set $V \setminus (U \cup W)$ is a
clique in colour $3$. Additionally, there is at least one triple of colour
$3$ in $U$ (and in $W$, but it may not be disjoint from those in $U$), so if
we take a maximal matching in colour $3$ extending an exhausting of $V
\setminus (U \cup W)$, we end up with at most $2 + \vert U \vert + \vert W
\vert - \vert U \cap W \vert -3 \leq 10$ vertices not covered by this
matching and consequently only containing triples of colours $1$ and $2$. We
distinguish between three possibilities for the number $m$ of vertices left
out:
\begin{itemize}
\item $m \leq 8$ and $m \neq 6$. By the theorem of Alon, Frankl and Lov\'asz
(\cite{Alonetal}) there is an almost perfect monochromatic matching in this
$2$-coloured subgraph: $3$ vertices needed for $1$ triple, $7$ for two
triples. Adding it to the initial colour $3$ matching, we obtain a near
perfect matching of $V$ in $2$ colours.

\item $m=6$ or $m=9$. In this case $\vert V \vert$ is a multiple of $3$, so
either it is at most $12$, in which case we apply Theorem \ref{thm:kozos},
or $\vert V \vert$ is at least $15$, hence the prediction \eqref{eqn:n2k}
gives a size at least one less than that of a perfect matching. In this
latter case, the result of \cite{Alonetal} is sufficient (a size $1$
matching for $m=6$ and a size $2$ matching for $m=9$).

\item $m=10$. Here all of our estimates have to be sharp, that is, $\vert U
\cap W \vert = 1$ and we must have $2$ vertices from $V \setminus (U \cup
W)$ and $8$ vertices from $U \cup W$ not covered by the matching in colour
$3$. If choosing a different maximal matching in colour $3$ leads to a
different case, we are done, so we may assume that no matter which $2$
vertices $a$ and $b$ of $V \setminus (U \cup W)$ are left out from the
initial matching, there do not exist $2$ disjoint triples of colour $3$ in
$U \cup W \cup \{ a,b \}$. But any vertex in $U \cup W \setminus (U \cap W)$
lies in the complement of a triple of colour $3$ - the elements of $U
\setminus W$ miss the colour $3$ triple in $W$ and vice versa. Therefore any
vertex in $U \cup W \setminus (U \cap W)$ together with any two vertices in
$V \setminus (U \cup W)$, and any two vertices in $U \setminus W$ (or $W
\setminus U$) together with any vertex in $V \setminus (U \cup W)$, give a
hyperedge of colour $1$ or $2$.
\par
If now $\vert V \vert \geq 31$, we can cover all of $V \setminus (U \cap W)$
(at most $30$ vertices) by at most $10$ such hyperedges (adding a suitable
splitting of $W$ or applying Theorem \ref{thm:kozos} if $\vert V \vert
=13$). If, on the other hand, $\vert V \vert \geq 32$, then the formula
\eqref{eqn:n2k} predicts a matching at least $1$ less than a near perfect
one. Such a matching can be found with direct application of \cite{Alonetal}
to the $10$-vertex remainder as before.
\end{itemize}
In all three cases we arrive at a matching in $2$ colours of size at least
that predicted by \eqref{eqn:n2k}, finishing the proof of Theorem
\ref{thm:main}.
%% \end{proof}

\end{document}